\documentclass[11pt]{article}
\usepackage{latexsym,amsmath,amsthm,amssymb,graphicx,mathrsfs,amsfonts,yfonts,calrsfs,eufrak,bm,stackrel}
\usepackage{tikz}
\usepackage{relsize}
\usepackage{cite}
\usepackage{mathtools}
\DeclareFontEncoding{LS1}{}{}
\DeclareFontSubstitution{LS1}{stix}{m}{n}
\DeclareSymbolFont{symbols3}{LS1}{stixbb}{m}{n}
\DeclareMathSymbol{\bigslopedvee}{\mathbin}{symbols3}{"A7}

        {\qquad\hspace*{\fill}$\square$\end{trivlist}}

        {\qquad\hspace*{\fill}\end{trivlist}}

\newcounter{fig}

\newcounter{clai}
\newcounter{exa}
\newcounter{theorem}[section]
\renewcommand{\thetheorem}{\arabic{section}.\arabic{theorem}}
\newcounter{nona}[theorem]
\newcounter{nonanona}[theorem]

\renewcommand{\thenona}{\Alph{nona}}
\renewcommand{\thenonanona}{\alph{nonanona}}

\newenvironment{theorem}{\begin{trivlist}\item[]\refstepcounter{theorem}%
        {\bf\thetheorem\ Theorem}\par\nobreak\noindent\sl\ignorespaces}{%
        \ifvmode\smallskip\fi\end{trivlist}}

\newenvironment{noname}{\begin{trivlist}\item[]\refstepcounter{nona}%
        {\bf (\thenona)\ \ \ }\nobreak\noindent\sl\ignorespaces}{%
        \ifvmode\smallskip\fi\end{trivlist}}

\newenvironment{lemma}{\begin{trivlist}\item[]\refstepcounter{theorem}%
        {\bf\thetheorem\ Lemma}\par\nobreak\noindent\sl\ignorespaces}{%
        \ifvmode\smallskip\fi\end{trivlist}}

\newenvironment{conjecture}{\begin{trivlist}\item[]%
        \refstepcounter{theorem}{\bf\thetheorem\ Conjecture}\par%
        \nobreak\noindent\sl\ignorespaces}{%
        \ifvmode\smallskip\fi\end{trivlist}}
\newenvironment{conjectureplus}[1]{\begin{trivlist}\item[]%
        \refstepcounter{theorem}{\bf\thetheorem\ Conjecture} %
        {\rm(\,#1\,)}\par\nobreak\noindent\sl\ignorespaces}{%
        \ifvmode\smallskip\fi\end{trivlist}}

\newenvironment{observation}{\begin{trivlist}\item[]\refstepcounter{theorem}%
        {\bf\thetheorem\ Observation\ }}{%
        \ifvmode\smallskip\fi\end{trivlist}}

        {\endlist\unskip}

        {\endtrivlist}
\newcommand{\wideitem}[1]{\leavevmode\hangindent\mathindent\noindent%
        \hbox to \mathindent{#1\hfil}\ignorespaces}

\renewcommand{\emptyset}{\mathchar"001F}
\newcommand{\eop}{\rule{1.4ex}{1.4ex}}

\newcommand{\subandsup}{%
\mathrel{%
\vcenter{\offinterlineskip\ialign{##\cr $\subset$\cr\noalign{\kern-.5pt}$\supset$ \cr }%
}%
}%
}%

\newcommand{\subandsupeq}{%
\mathrel{%
\vcenter{\offinterlineskip\ialign{##\cr $\subseteq$\cr\noalign{\kern-.5pt}$\supseteq$ \cr }%
}%
}%
}%

\newcommand{\C}{{\cal C}}

\newcommand{\Q}{{\cal Q}}

\newcommand{\parallelbowtie}{\ \mathrel{\bowtie \! \! \! \!\parallel}\ \ }

\newcommand{\sms}{\vspace*{.2in}}

\title{\bf Skew circuits and circumference in a binary matroid}

\author{Sean McGuinness \\ Dept. of Mathematics\\ Thompson Rivers University\\
McGill Road, Kamloops BC\\ V2C5N3 Canada\\ email: smcguinness@tru.ca}

\date{}

\begin{document}

\maketitle

\begin{abstract}
The {\bf circumference} of a matroid is defined to be the size of the largest circuit.  For disjoint sets $X$ and $Y$ in a matroid $M$ having ground set $E$ we define $\kappa_M(X,Y) = \min_{X \subseteq A \subseteq E-Y} r(A) + r(E-A) - r(M).$
Let $C_1$ and $C_2$ be two disjoint circuits where $r(C_1 \cup C_2) = r(C_1) + r(C_2)$ in a binary matroid $M$ having circumference $c$.  We show that for every nonnegative integer $k$, there is an integer $\alpha(k)$ such that if $\kappa_M(C_1,C_2) \ge \alpha(k)$, then $|C_1| + |C_2| \le 2c-k.$

\bigskip\noindent

\end{abstract}

\section{Introduction}
Many properties of cycles in graphs have their counterparts for circuits in matroids.  Our starting point for investigation, is understanding in what way the property of intersection of long cycles in graphs can be extended to matroids.  For graphs, connectivity plays a significant role in forcing long cycles to intersect and of particular interest here is an old  conjecture of Smith (see \cite{Bon}):

\begin{conjectureplus}{Smith, 1984}
If $C$ and $D$ are longest cycles in a $k$-connected graph, where $k\ge 2,$ then $C$ and $D$ meet in at least $k$ vertices.
\end{conjectureplus}

The conjecture remains unsolved, and perhaps the best known result related to it can be found in \cite{CheFauGou}, where it is shown that two longest cycles in a $k$-connected  
graph must meet in at least $ck^{\frac 35}$ vertices ( $c \approx 0.2615$).  The key tool here is Menger's theorem (see \cite{Bon}) which guarantees the existence of many vertex-disjoint paths between two edge-disjoint cycles in a highly connected graph. 
 In \cite{Gro}, the properties of longest cycles sharing $3,4,$ or $5$ vertices are studied and Smith's conjecture is verified for $k \le 6.$  

The above conjecture raises some natural questions for matroids.  For a matroid $M$ having at least one circuit, we define the {\bf circumference}, denoted $c(M)$ to be the size of a largest circuit in $M$.  A $k$-connected matroid is defined in the following way (see \cite{Oxl}).  Let $M$ be a matroid having rank function $r.$
For sets $X$ and $Y$, let $\sqcap_M(X,Y) = r(X) +r(Y) - r(X\cup Y).$
The {\bf connectivity function} for $M$ is the function $\lambda_M: 2^{E(M)} \rightarrow \mathbb{Z}$ where $\lambda_M(A) = r(A) + r(E(M) -A) - r(M);$ that is, $\lambda_M(A) = \sqcap_M(A, E(M) - A).$
For disjoint subsets $X$ and $Y$ we define $\kappa_M(X,Y) = \min_{X \subseteq A \subseteq E(M) - Y} \lambda_M(A)$ which we refer to as the {\bf linkage} between $X$ and $Y$.  We say that $X$ and $Y$ are $\mathbf{k}${\bf-linked} if $\kappa_M(X,Y) \ge k.$  For a subset $A \subset E(M),$ we say that $A$ (or the partition $(A, E(M) - A)$) is $\mathbf{k}${\bf -separating} if 
$\lambda_M(A) < k$.  The matroid $M$ is said to be $\mathbf{k}${\bf-connected} if it has no $(k-1)$-separating partition $(A, E(M)-A)$ where $\min \{ |A|, |E(M) - A| \} \ge k-1.$  It follows from the definition of a $k$-connected matroid, that if $|E(M)| \ge 2k-2,$ then all circuits must have size at least $k.$  When $|E(M)| < 2k-2,$  
it is possible that $M$ can have small circuits and as such, the linkage between two disjoint circuits can be less than $k$.  

Two sets $X$ and $Y$ in matroid are said to be {\bf skew} if $r(X) + r(Y) = r(X\cup Y).$

One interpretation of Smith's conjecture for a graphic matroid $M$ is that it asserts that if $C_1$ and $C_2$ are two $k$-linked disjoint circuits of size $c(M)$, then $\sqcap_M(C_1,C_2) \ge  k-1.$   Conversely, if $C_1$ and $C_2$ are disjoint $k$-linked circuits for which $\sqcap_M(C_1,C_2) < k-1,$ then $|C_1| + |C_2| < 2c(M).$  Based on this, one might guess that for matroids in general if $C_1$ and $C_2$ are $k$-linked skew circuits in a matroid $M$, then $|C_1| + |C_2|$ should be smaller than $2c(M)$ by an amount depending on $k.$ This has been verified in certain special cases.
%
In \cite{Mcg}, it was shown that if $M$ is a $k$-connected regular matroid with circumference at least $k$ and $C_1$ and $C_2$ are skew circuits, then $|C_1| + |C_2| \le 2(c(M)-k+1).$
In \cite{McmReiWei}, the authors showed that if $C_1$ and $C_2$ are skew circuits in a connected matroid, then $|C_1| + |C_2| \le 2c(M)-2.$  The case of large circuits in cographic matroids is studied in \cite{She} where the intersection of largest bonds in $k$-connected graphs is studied.   Generally, finding similar bounds for matroids of higher connectivity is hard.  We conjecture that a bound of the form $2c(M)-k$ exists when the linkage between two skew circuits is sufficiently large.  Specifically, we conjecture:

\begin{conjecture}
Let $C_1$ and $C_2$ be skew circuits  in a matroid $M$ having circumference $c.$ Then for all nonnegative integers $k$ there exists an integer $\alpha(k)$, depending only on $k$,  such that if $C_1$ and $C_2$ are $\alpha(k)$-linked, then 
$|C_1| + |C_2| \le 2c -k.$ \label{con1}
\end{conjecture}

The main result of this paper is that we verify this conjecture for binary matroids.

\begin{theorem}
Let $C_1$ and $C_2$ be skew circuits in a binary matroid $M.$ Then for all nonnegative integers $k$ there exists an integer $\alpha(k)$, depending only on $k$,  such that if $C_1$ and $C_2$ are $\alpha(k)$-linked,  then 
$|C_1| + |C_2| \le 2c(M) -k.$ \label{the-main}
\end{theorem}

We note that it suffices to prove the theorem for any infinite sequence of positive integers, in particular when $k$ is an even integer.  The existence of the constant $\alpha(k)$ follows implicitly from the proof of the above theorem and we give no explicit value other than to say that it is large.  

\subsection{Notation}

For the most part, we shall follow the definitions and notation for matroids given in \cite{Oxl} with the exception of some notation for sets.
For finite sets $X$ and $Y$ of a universal set $U$ we write $X \parallel Y$ if $X \cap Y = \emptyset$ and we write $X \subandsup Y$ (resp. $X \subandsupeq Y$) if either $X \subset Y$ or $Y \subset X$ (resp. $X \subseteq Y$ or $Y \subseteq X$). 
We write $X \bowtie_U Y$ if $U -X \parallel U - Y,$ where we drop the index $U$ when it is implicit.  We write $X \parallelbowtie Y$ if either $X\parallel Y$ or $X \bowtie Y.$

For a set $X$ and a nonnegative integer $k$, we let ${\binom Xk}$ denote the set of all $k$-subsets of $X.$  For finite nonempty subsets of integers $U$ and $V$, we write $U < V$ if $\max \{ u\in U \} < \min \{ v \in V \}.$

For a set $X$ and elements $x,y$ we will often write $X - x$ in place of $X - \{ x \}$ and $X - x + y$ in place of $(X - \{ x \} ) \cup \{ y \}$.  More generally, for elements $x_1, \dots ,x_k$ and $y_1, \dots ,y_\ell$, $X - x_1 - \cdots - x_k + y_1 + \cdots + y_\ell$ will denote the set $(X - \{ x_1, \dots ,x_k \} ) \cup \{ y_1, \dots ,y_\ell \}.$

For a positive integer $k$, we let $[k] = \{ 1, \dots ,k \}$ and we let $[k]^e$ (resp. $[k]^o$) denote the set of even (resp. odd) integers in $[k].$  For integers $k < \ell$, we let 
$[k,\ell] = \{ k, k+1, \dots , \ell \}$ and we let $[k,\ell]^e$ (resp. $[k,\ell]^o$) denote the set of even (resp. odd) integers in $[k,\ell].$

\section{An overview of the proof}
If for some $i\in \{ 1,2 \}$,  $|C_i| \le c(M)-k,$ then the conclusion of Theorem \ref{the-main} is immediate.  As such, one can assume that for $i=1,2,$ $|C_i| > c(M) -k.$
The proof of Theorem \ref{the-main} boils down to finding a certain pair of circuits of which there are two possible scenarios. 
 
We define a {\bf cycle} of $M$ to be a disjoint union of circuits, including the empty set.
We let $\C_M$ denote the set of all cycles in $M$ and for a subset $X \subseteq E(M)$, we let $\C_M(X)$ be the set of cycles contained in $X.$  For cycles $C_1$ and $C_2$, we define addition $C_1 + C_2 := C_1 \triangle C_2$ (where $\triangle$ denotes the symmetric difference operation).  Since the matroids considered in this paper are binary, $\C_M$ is seen to be closed under this addition.

Suppose $M$, $C_1,$ $C_2$, and $k$ are as in the statement of Theorem \ref{the-main}.  To prove this theorem, it will suffice to show that at least one of two things must occur:

\begin{itemize}
\item[{\bf S1})] There is a cycle $C$ where $C+C_1$ and $C+ C_2$ are circuits and\\ $|C - (C_1 \cup C_2)| \ge \frac k2$.
\item[{\bf S2})] There are disjoint cycles $C_1'$ and $C_2'$ where $C_1 + C_2 + C_i',\ i = 1,2$ are circuits and $|(C_1 \cup C_2) + (C_1' \cup C_2')| \ge k.$ 
\end{itemize}

If S1) occurs, then we have $$2c(M) \ge |C+ C_1| + |C + C_2| = |C_1| + |C_2| + 2|C - (C_1 \cup C_2)|| \ge |C_1| + |C_2| + k.$$  If S2) occurs, then we have $$2c(M) \ge \sum_{i=1}^2 |C_1 + C_2 + C_i'| = |C_1| + |C_2| + |(C_1 \cup C_2) + (C_1' \cup C_2')| \ge |C_1| + |C_2| + k.$$
%
To find such cycles, we first use a version of Tutte's Linking Lemma to reduce the problem to a minor $N$ where $C_i,\ i = 1,2$ are skew circuits in $N$ and $C_1 \cup C_2$ spans $N.$ Furthermore, the matroid $N$ will also have the property that $\kappa_N(C_1,C_2) = |E(N) - (C_1 \cup C_2)|.$   It will suffice to prove the theorem with $N$ in place of $M.$  Assuming $X = E(N) - (C_1 \cup C_2) = \{ x_1, \dots ,x_t \}$, for each $i\in[t]$ we pick a circuit $D_i$ in $(C_1 \cup C_2) +x_i.$  The cycles we are looking for to satisfy S1) or S2) will be of the form $D_I$ where $I \subset [t]$ and $D_I = \sum_{i\in I} D_i.$  

To begin with, suppose there exist constants $\alpha_i(k),\ i = 1,2$ depending only on $k$ with the property that if $I \subset [t]$ has size at least $\alpha_i,$ then there is a subset $J \in {\binom Ik}$ for which $D_J + C_i$ is a circuit.  Then one can use Ramsey's theorem for hypergraphs to show that there is a constant $\alpha(k)$ depending only on $k$ such that if $I \subset [t]$ has size at least $\alpha$, then for some $J \in {\binom Ik}$ both $D_J + C_1$ and $D_J + C_2$ are circuits.  In this case, one can find a circuit $C$ as described in S1). 
Thus we assume that no such constants $\alpha_i,\ i = 1,2$ exist and in particular we may assume that for an integer $p$ chosen as large as we want, for all 
$I \in {\binom {[p]}k},$ $D_I + C_1$ is not a circuit.   For the remainder of the proof, we shall show that there is a pair of cycles $C_i',\ i = 1,2$ as described in S2).   
To do this, we use a theorem of Balogh and Bollob\'{a}s \cite{BalBol} regarding unavoidable configurations in $0,1$-matrices.  This theorem asserts that for any integer $\ell \ge 1,$ if a simple $0,1$-matrix has a large enough number of rows, then one can find an induced submatrix which, after permuting its rows and columns,  is either $I_\ell$, $I_\ell^c$, or $T_\ell$; that is, the $\ell \times \ell$ identity matrix, its complement, or the $\ell \times \ell$ lower-triangular matrix with $1$'s below or on the diagonal. We apply this theorem to the matrices $A_i,\ i = 1,2$ where $A_i$ is the $0,1$-matrix whose rows correspond to the binary indicator vectors for the subsets $D_j \cap C_i, \ j \in [p].$  
This will enable us to show that provided $p$ is large enough, one can construct pairwise disjoint sets $U_i \subset [p],\ i \in [4]$ such that for $i=1,2,3,$  $D_{U_i} \cap C_1 \subset D_{U_{i+1}} \cap C_1$ and the sets $D_{U_i}\cap C_2,\ i \in [4]$ are pairwise disjoint.  Furthermore, for all $1 \le i < j \le 4,\ |U_i \cup U_j| \ge k.$  It can now be shown that $D_{U_1 \cup U_2}$ and $D_{U_{3} \cup U_4}$ are cycles satisfying S2).


\section{Exploiting linkage}

We shall make use of the following result of Tutte \cite{Tut} known as {\bf Tutte's linking lemma},  which, in the case of graphic matroids, implies Menger's theorem. We shall use a stronger version of this theorem found in \cite{GeeGerWhi}.

\begin{lemma}
Let $M$ be a matroid and let $X,Y$ be disjoint subsets of elements. Then there is a minor $N$ of $M$ such that, $E(N) = X \cup Y$, $\sqcap_N(X,Y) = \kappa_N(X,Y) = \kappa_M(X,Y)$, and $N \big| X = M \big| X,\ N \big| Y = M \big| Y$.
\label{lem1}
\end{lemma}

For our purposes, we shall need a stronger version of the above lemma which can be found in \cite{GeeNelWal}:

\begin{lemma}
Let $M$ be a matroid and let $m\ge 1$ be an integer.  Suppose that for a subset $X \subset E(M)$ and subsets $Y_i \subset E(M),\ i \in [m],$ we have $Y_1 \subseteq Y_2 \subseteq \cdots \subseteq Y_m \subseteq E(M) - X.$  Then $M$ has a minor $N$ with ground set $X \cup Y_m$ such that 
 for all $i\in [m]$, $\kappa_N(X,Y_i) = \kappa_M(X, Y_i)$, and $N \big| X = M \big| X$ and $N\big| Y_1 = M \big| Y_1.$\label{lem2}
\end{lemma}

The first step in the proof of Theorem \ref{the-main} is to reduce the proof to a matroid $N$ where $C_1 \cup C_2$ spans $N$ and $\kappa_N(C_1,C_2) = \kappa_M(C_1,C_2).$
This can be done using the following lemma.

\begin{lemma}
Let $X$ and $Y$ be disjoint sets of a matroid $M.$  Then there exists a minor $N$ such that $X \cup Y \subseteq E(N),$ $X \cup Y$ spans $N$, $\sqcap_N(X,Y) = \sqcap_M(X,Y)$, $\kappa_N(X,Y) = \kappa_M(X,Y)$
and $N\big| X = M\big| X,$ $N\big| Y = M \big| Y.$
\label{lem4}
\end{lemma}

\begin{proof}
By induction on $|E(M) - X-Y|.$  The lemma is clearly true if $X \cup Y$ spans $M$ (since taking $N=M$ will suffice).  Thus we may assume that $X \cup Y$ does not span $M$ and that the lemma is true for all matroids $M'$ where $|E(M') - X - Y| < |E(M) - X - Y|.$   Let $e \in E(M) - \mathrm{cl}_M(X \cup Y)$ and let $Y_1 = Y$ and $Y_2 = E(M) - X - e.$ By Lemma \ref{lem2}, there exists a minor $M_1$ of $M$ with ground set $X \cup Y_2 = E(M) - e$ such that 
$\kappa_{M_1}(X, Y) = \kappa_{M}(X,Y).$  We note either $M_1 = M/e$ or $M_1= M\backslash e.$ 
Since $e\not\in \mathrm{cl}_M(X \cup Y),$ we have $\sqcap_{M_1}(X,Y) = \sqcap_M(X,Y)$, $M_1\big| X = M\big| X,$ $M_1\big| Y = M\big| Y$.  By induction, there exists a minor $N$ of $M_1$ containing $X\cup Y$ such that $X \cup Y$ spans $N$ and $\kappa_N(X,Y) = \kappa_{M_1}(X,Y) = \kappa_{M}(X,Y),$ $\sqcap_N(X,Y) = \sqcap_{M_1}(X,Y) = \sqcap_M(X,Y).$  Moreover, $N \big| X = M_1 \big| X = M \big| X$ and $N \big| Y  = M_1 \big| Y = M \big| Y.$ 
\end{proof}

\subsection{Reducing the problem to a minor $N$}

 The advantage of the previous lemmas is that they allow us to reduce the proof of Theorem \ref{the-main} to a minor $N$ of $M$ for which
$\kappa_N(C_1,C_2) = \kappa_M(C_1,C_2)$ and $\sqcap_N(C_1, C_2) = \sqcap_M(C_1, C_2) = 0.$ 

For the remainder of this paper, we shall let $C_i = \{ e_{i1}, \dots e_{in_i} \},\ i = 1,2$ be two disjoint circuits in a binary matroid $M$ where $\sqcap (C_1, C_2) =0.$  By Lemma \ref{lem4}, there is a minor $M'$ of $M$ containing $C_1 \cup C_2$ where $C_1 \cup C_2$ spans $M'$, $\kappa_{M'}(C_1, C_2) = \kappa_{M}(C_1, C_2),$ $\sqcap_{M'}(C_1, C_2) = \sqcap_M(C_1, C_2) =0,$ and $M'\big| C_1 = M\big| C_1,\ M' \big| C_2 = M \big| C_2.$  
We may assume that for $i =1,2,$ $\mathrm{cl}_{M'}(C_i) = C_i.$  
To see this, it suffices to show for $e\in \mathrm{cl}_{M'}(C_1) - C_1$ that $\kappa_{M'\backslash e}(C_1, C_2) = \kappa_{M'}(C_1, C_2).$ 
Clearly $\kappa_{M'\backslash e}(C_1, C_2) \le \kappa_{M'}(C_1, C_2)$, noting that $r(M'\backslash e) = r(M').$   Let $A \cup B$ be a partition of $E(M'\backslash e)$ where $C_1 \subset A$ and $C_2 \subset B.$
Let $A' = A+e$ and $B' = B.$  Then $A' \cup B'$ is a partition of $E(M)$ where $C_1 \subset A'$ and $C_2 \subset B'$ and $r(A') + r(B') = r(A) + r(B).$  Thus  $\kappa_{M'}(C_1,C_2) \le \kappa_{M'\backslash e}(C_1,C_2)$ and hence equality must hold. 

 By Lemma \ref{lem1}, there exists a minor $N' = M'/X\backslash Y$ of $M'$ such that $E(N') = C_1 \cup C_2$, and $\sqcap_{N'}(C_1, C_2) = 0$ and $\kappa_{N'}(C_1, C_2) = \kappa_{M'}(C_1,C_2)$ and $N'\big| C_i = M'\big| C_i,\ i = 1,2.$  Here we may assume that $X$ is independent since if contracting some elements of $X$ results in a loop $e\in X$, then deleting or contracting $e$ yields the same matroid, in which case one could assume that $e\in Y.$  
     
Let $N = M'\backslash Y = M' \big| C_1 \cup C_2 \cup X$ and let $\kappa_N(C_1, C_2) = t.$

\begin{observation}
For $i=1,2$, $C_i$ is the only circuit of $N$ in $C_i \cup X$ and for all subsets $X' \subseteq X,$ $r_N(C_i \cup X') = r_N(C_i) + |X'|.$ \label{obs-onlycirc}
\end{observation}

\begin{proof}
Suppose to the contrary that for some $i$, $C_i \cup X$ contains a circuit $C$ where $C \ne C_i.$
Then $C' = C - X$ is a cycle in $N' = N/X$ where $C' \subset C_i.$  However, this contradicts the assumption that $N' \big| C_i = M\big| C_i.$
It now follows that for all $i\in \{ 1,2 \}$ and for all subsets $X' \subseteq X,$ $r_N(C_i \cup X') = |C_i| - 1 + |X'| = r_N(C_i) + |X'|.$
\end{proof}

\begin{observation}
 $t = \kappa_N(C_1, C_2) = |X|.$\label{obs-X=t}
 \end{observation}
 
\begin{proof} Let $A \subset E(N)$ where $C_1 \subseteq A \subseteq E(N) - C_2.$  Let $X_1 = A \cap X$ and let $X_2 = X-A.$  Then 

\begin{align*}
\lambda_N(A) &= r_N(A) + r_N(E(N) - A) - r(N)\\ &= r_N(C_1) + |X_1| + r_N(C_2) + |X_2| - r(N)= |X|.\end{align*} 
The second equality follows from Observation \ref{obs-onlycirc} and the fact that $C_i,\ i = 1,2$ are skew circuits for which $C_1 \cup C_2$ spans $N.$
Thus it follows that $t = \kappa_N(C_1, C_2) = |X|.$ 
\end{proof}

For the remainder of this paper, we let $X = \{ x_1, \dots ,x_t \}$ and let $N$ be as described above. To prove Theorem \ref{the-main}, it will suffice to prove it for the matroid $N.$

For all subsets $A \subseteq E(N)$ and for all $j\in \{ 1,2 \},$ we shall let $A^j$ denote the set $A \cap C_j.$

\subsection{The circuits $D_i$}
Given that $C_1 \cup C_2$ spans $N$, it follows that for all $i\in [t]$ there is a circuit in $N$ which contains $x_i$ and which is contained in $(C_1 \cup C_2) + x_i.$  We let $D_i$ be one such circuit.    
For all subsets $\emptyset \ne I\subseteq [t],$ let $D_I = \sum_{i\in I}D_j$.  It is seen that $D_I$ is a cycle in $\C_N - \C_N(C_1 \cup C_2).$   Since for $j=1,2,$ $C_j$ is the only circuit in $C_j \cup X$ (by Observation \ref{obs-onlycirc}), it follows that for $j= 1,2,$ $C_j \not\subseteq D_I$.  For if for $C_j \subseteq D_I$, then $D_I + C_j$ is a cycle where $X \subset D_I + C_j \subseteq C_{3-j} \cup X$.  It would then follow by Observation \ref{obs-onlycirc} that $D_I + C_j = C_{3-j},$ which is impossible.
Thus we have the following:

\begin{observation}For $\emptyset \ne I \subseteq [t],$  and for all $j\in \{ 1,2 \},$ $\emptyset \ne D_I^j \subset C_j.$  Furthermore, no non-empty cycle of $\C_N(C_1 \cup C_2) = \{ \emptyset, C_1, C_2, C_1 \cup C_2 \}$ is contained in $D_I + C_j.$ \label{obs-noempty}\end{observation}

Note that if $I$ and $I'$ are distinct, nonempty subsets of $[t]$, then for $j=1,2,$ $D_I^j  \ne D_{I'}^j;$ for if  $D_I^j  = D_{I'}^j,$ then we would have for $I'' = I \triangle I',$ $D_{I''}^j  = \emptyset,$ contradicting the previous observation.

\subsection{When $D_I + C_j$ is not a circuit}

Suppose $I \subseteq [t]$ and $j\in \{ 1,2 \}.$  If for some $j\in [2],$ $D_I + C_j$ is not a circuit, then it breaks into two cycles and we will show that there is a specific structure in this case.

\begin{lemma}
Let $I \subseteq [t]$ and let $j\in \{ 1,2 \}.$ If $D_I + C_j$ is not a circuit, then there is a partition $I = I_1 \dot{\cup} I_2$ such that $D_{I_1}^j  \subandsup D_{I_2}^j$ and $D_{I_1}^{3-j} \parallelbowtie D_{I_2}^{3-j}$. Furthermore, if $D_I + C_1 + C_2$ is not a circuit, then there is a partition $I = I_1 \dot{\cup} I_2$ such that for $i =1,2,$ $D_{I_1}^i \subandsup D_{I_2}^i.$ \label{lem-useful2}
\end{lemma}

\begin{proof}
It suffices to prove the first statement for $j=1.$  Suppose that $D_I + C_1$ is not a circuit.
By Observation \ref{obs-noempty},
no non-empty cycle of $\C_N(C_1 \cup C_2) = \{ \emptyset, C_1, C_2, C_1 \cup C_2 \}$ is contained in $D_I + C_1.$   Since $D_I + C_1$ is not a circuit, there are disjoint cycles $G_i \in \C_N - \C_N(C_1 \cup C_2),\ i =1,2$ for which
$D_I + C_1 = G_1 \dot{\cup} G_2.$  Then there is a partition $I_1 \dot{\cup} I_2 = I$ of $I$ such that for $i =1,2,$  $I_i = \{ j \in I \ \big| \ x_j \in G_i \cap X \}.$
Since  $D_{I_1} \cap X = G_1 \cap X,$ it follows that for some $H \in \C_N(C_1 \cup C_2),$ $D_{I_1} = G_1 + H$ and  
$D_{I_2} = G_2 + H + C_1.$   It can be easily checked that for all $H\in \C(C_1 \cup C_2),$
$D_{I_1}^1  \subandsup D_{I_2}^1$ and $D_{I_1}^{2} \parallelbowtie D_{I_2}^{2}$.  For example, if $H = \emptyset,$ then $D_{I_1} = G_1,$ $D_{I_2} = G_2 + C_1$ and $D_{I_1}^1 \subset D_{I_2}^2$ and $D_{I_1}^1 \parallel D_{I_2}^2.$ 
%
%
%
%
%
%
%

To prove the second statement, suppose $D_I + C_1 + C_2$ is not a circuit.  Let $i \in I$ and replace $D_i$ with $D_i' = D_i + C_1.$ Let $D_I' = D_i' + D_{I-i}.$  Then $D_I' + C_2$ is not a circuit and hence by the first part, there is a partition $I = I_1 \dot{\cup} I_2$ where  $(D'_{I_1})^1 \parallelbowtie (D'_{I_2})^1$ and $(D'_{I_1})^2 \subandsup (D'_{I_2})^2.$  Here we assume $i\in I_1$ and $D'_{I_1} = D_i' + D_{I_1-i}$ and $D'_{I_2} = D_{I_2}.$
Given that $(D'_{I_1})^1 = D_{I_1}^1 + C_1,$ $(D'_{I_2})^1 = D_{I_2}^1,$ and $(D'_{I_k})^2 = D_{I_k}^2,\ k = 1,2$, it follows that $D_{I_1}^1 \subandsup D_{I_2}^1$ and $D_{I_1}^2 \subandsup D_{I_2}^2.$  This proves the second part.
\end{proof}

\section{Ramsey's Theorem and unavoidable configurations in matrices}

A key component in the proof of Theorem \ref{the-main} involves an application of {\bf Ramsey's theorem} in its form for hypergraphs (see \cite{MubSuk}).
The {\bf complete} $\mathbf{r}${\bf -uniform hypergraph} on a set of vertices $S$ is the hypergraph whose edges are all the $r$-subsets of $S.$  For $|S| = n,$ we let $K_n^r$ denote the complete $r$-hypergraph on $n$ vertices.  Ramsey's theorem for hypergraphs can be stated as follows:

\begin{theorem}
Suppose $s_1, \dots ,s_k$ are positive integers, all at least $r.$  There is an integer $R^r(s_1, \dots ,s_k)$ such that if $n \ge R^r(s_1, \dots ,s_k),$ then for any $k$-colouring of the edges of $K_n^r$ with colours $1, \dots ,k$, there exists $i \in [k]$ and a complete $r$-uniform subhypergraph $K_{s_i}^r$  on $s_i$ vertices, all of whose edges have colour $i.$
\label{the-Ram}
\end{theorem}

As an application of the above theorem, we shall show that a circuit $C$ as described in S1) exists unless for some $j\in \{ 1,2 \}$ there are large subsets $I \subseteq [t]$ (where $t = |X|$) such that for all $J \in {\binom {I}k}$, $D_J + C_j$ is not a circuit.  This follows from the next lemma.

\begin{lemma}
Suppose that there are constants $\alpha_i(k), \ i = 1,2$, depending only on $k$, such that for all subsets $I \subseteq [t],$ if $| I | \ge \alpha_i(k),$ then there is a $k$-subset $J \subseteq I$ such that $D_J + C_i$ is a circuit.
Then there is a constant $\alpha(k)$ such that for all subsets $I \subseteq [t]$ where $| I | \ge \alpha(k),$ there is a  subset $J \subseteq I$ such that all $J' \in {\binom Jk}$ and for all $i\in [2],$ $D_J + C_i$ is a circuit.
\label{obs5}
\end{lemma}

\begin{proof}
Let $I \subseteq [t]$ and let $I' \subseteq I.$  Suppose we colour the $k$-subsets of $I'$ blue or red in such a way that a $k$-subset $J$ is coloured blue if $D_J + C_1$ is a circuit;  otherwise, it is coloured red.
By Theorem \ref{the-Ram}, there is a constant $R^k(\alpha_1, \alpha_1)$ such that if $|I'| \ge R^k(\alpha_1, \alpha_1)$ then for any $2$-colouring of the uniform $k$-hypergraph on $I'$, there is an $\alpha_1$-subset $I'' \subseteq I'$ whose associated $k$-hypergraph is monochromatic. By assumption, $I''$ must contain at least one blue hyperedge, and hence each hyperedge of $I''$ is blue.
Let $\alpha = R^k(R^k(\alpha_2, \alpha_2), \alpha_1)$ and suppose that $| I | \ge \alpha.$  Then either there exists an $R^k(\alpha_2, \alpha_2)$-subset $I' \subseteq I$ all of whose $k$-subsets are blue, or there exists an $\alpha_1$-subset all of whose $k$-subsets are coloured red.  By our assumptions on $\alpha_1$, the latter is impossible and thus the former holds.  Let $I' \subseteq I$ be an $R^k(\alpha_2, \alpha_2)$-subset  all of whose $k$-subsets are blue.  Now if we colour the $k$-subsets $J$ of $I'$ green or yellow, depending on whether $D_J + C_2$ is a circuit or not, it follows by Theorem \ref{the-Ram} that there is an $\alpha_2$-subset of $I'' \subseteq I'$ whose $k$-uniform hypergraph is monochromatic.  By our assumptions on $\alpha_2$, all the $k$-subsets of $I''$ must be green. Thus every $k$-subset $J$ of $I''$ is both blue and green and thus for $i=1,2,$ $D_J + C_i$ is a circuit.   
\end{proof}

 By the above lemma, we may assume for the remainder that for an integer $p$, chosen as large as needed, for all $I \in {\binom {[p]}k}$,  $D_I + C_1$ is not a circuit.  Consequently, for all $I \in {\binom {[p]}k},$ there is a partition $I = I_1 \dot{\cup} I_2$ such that $D_I + C_1 = G_1 + G_2$ where $G_i,\ i = 1,2$ are disjoint cycles such that for $i=1,2,$ $D_{I_i} \cap X = G_i \cap X.$  Importantly, we note that since $D_I +C_1$ is the disjoint union of the cycles $G_{I_1}$ and $G_{I_2}$, one can show (following the proof of Lemma \ref{lem-useful2}) that $D_{I_1}^1  \subandsup D_{I_2}^1$ and $D_{I_1}^{2} \parallelbowtie D_{I_2}^{2}$.

\subsection{A theorem on unavoidable configurations in matrices}\label{sec-forbid}

A matrix is {\bf simple} if it has no repeated rows.  For matrices $A$ and $B$ we write $B \prec A$ if $B$ can be obtained from  a submatrix of $A$ by permuting certain rows and columns of the submatrix.

For a positive integer $\ell \ge 1,$  $I_\ell$ will denote the $\ell \times \ell$ identity matrix and $I_\ell^c$ will denote the $0,1$-matrix which is its complement.  We let $T_\ell$ (resp. $T_\ell'$) denote the $\ell \times \ell$\ $0,1$-matrix having a $1$ in position $(i,j)$ if and only if $j\le i$ (resp. $i\le j$).  We say that two square $0,1$-matrices $B_1$ and $B_2$ have the same {\bf type} if for some integers $\ell_1$ and $\ell_2$, $(B_1, B_2) \in \{ (I_{\ell_1}, I_{\ell_2}),  (I_{\ell_1}^c, I_{\ell_2}^c), (T_{\ell_1}, T_{\ell_2}) \}.$

We shall make use of the following theorem of Balogh and Bollob\'{a}s \cite{AnsLu, BalBol }:

\begin{theorem}
Let $A$ be a simple $m\times n$ $0,1$-matrix. If $m\ge (2\ell)^{2^\ell}$, then there exists $B \in \{ I_\ell, I_\ell^c, T_\ell \}$ for which $B \prec A.$ \label{the-BalBol}  
\end{theorem}

For all positive integers $i$, let $\beta(i) = (2i)^{2^i}.$  The above theorem will allow us to choose a large subset of the circuits of
$\{ D_1, \dots ,D_p \}$ in a favorable way.

\subsubsection{The matrices $A,A'$, $A_i,B_i$, and $A_i'$}\label{sec-matrices}  
Recall that for $i=1,2,$ $C_i = \{ e_{i1}, \dots , e_{in_i} \}$.  For $k=1,2,$ we define a $p \times n_k$ \ 
$0,1$-matrix $A_k=[a_{ij}^{(k)}]$ where $a_{ij}^{(k )} =1$ if and only if $e_{k j} \in D_i^k = D_i \cap C_k.$  We observe that the matrix $A_k$ is simple.  Let $A$ be the $p\times (n_1 + n_2)$ matrix obtained by concatenating $A_1$ with $A_2$; that is, $A_1$ (resp. $A_2$) corresponds to the submatrix of $A$ formed by the first $n_1$ (resp. last $n_2$) columns of $A.$  

Suppose that $p \ge \beta(\beta(q)) = (2\beta(q))^{2^{\beta(q)}}.$  Since $A_1$ is simple, it follows by Theorem \ref{the-BalBol} that for some $B_1 \in  \{ I_{\beta(q)}, I_{\beta(q)}^c, T_{\beta(q)} \}$,  $B_1 \prec A_1.$ 
By permuting the rows of $A$ and its first $n_1$ columns, we may assume that $B_1$ is the submatrix of $A$ whose entries occupy the first $\beta(q)$ rows and first $\beta(q)$ columns of $A.$  Now let $A'$ be the $\beta(q) \times (n_1 +n_2)$ submatrix of $A$ corresponding to the first $\beta(q)$ rows of $A$ and let $A_1'$ (resp. $A_2'$)  be the submatrix of $A'$ formed by the first $n_1$ (resp. last $n_2$) columns of $A'$, noting that $A_i',\ i =1,2$ are simple since $A_i,\ i = 1,2$ are.  
It follows by Theorem \ref{the-BalBol} that there is a matrix $B_2 \in \{ I_q, I_q^c, T_q \}$ where $B_2 \prec A_2'.$

\section{The case where $B_1$ and $B_2$ are of the same type}\label{sec-sametype}

In this section, we will show that for $k$ even and $q$ large enough, $B_1$ and $B_2$ are not of the same type.  

\begin{noname}
Suppose $k$ is even and $B_1$ and $B_2$ are of the same type.  Then $B_1 = T_{\beta(q)}$ and $B_2 = T_q.$ \label{nona-B1=B2}
\end{noname}

\begin{proof}
Suppose that $B_1 \in \{ I_{\beta(q)}, I_{\beta(q)}^c \}$ and $B_2 \in \{ I_q, I_q^c \}.$  It is seen that one can permute the rows of $A'$ and the columns of $A_1'$ and/or the columns of $A_2'$, so as to obtain a matrix $A''$ and submatrices $A_i'', \ i = 1,2$ corresponding to $A_i',\ i = 1,2$, where for $i=1,2$, the submatrix $B_i'$ formed by the first $q$ columns of $A_i''$ is one of the matrices $I_q$ or $I_q^c$ and has the same type as $B_i.$  In this case, we will assume that the elements of $C_1 \cup C_2$ and circuits $D_i, \ i \in [q]$ are indexed so that $e_{ij}$ is the element corresponding to the $j$'th column of $A_i''$ and
$D_i$ corresponds to the $i$'th row of the matrix $A''.$

 Given that for $i=1,2,$ $B_i'$ has the same type as $B_i$, we have that $B_1' = B_2'.$ 
 Let $J \in {\binom {[q]}k}.$  By assumption, $D_J + C_1$ is not a circuit.  Thus 
there is a partition $J = J_1 \dot{\cup} J_2$ of $J$ such that  $D_{J_1}^1 \subandsup D_{J_2}^1$ and  $D_{J_1}^2 \parallelbowtie D_{J_2}^2$.  When we restrict to the elements of $E_1 = \{ e_{11}, \dots , e_{1q} \}$ in $C_1$ and the elements  of $E_2 = \{ e_{21}, \dots , e_{2q} \}$ in $C_2,$ the same relations hold (almost).  That is, 

\sms

i) \   $D_{J_1} \cap E_1 \subandsupeq D_{J_2} \cap E_2$ \ \ and \ \  ii)\ $D_{J_1} \cap E_1  \parallelbowtie D_{J_2} \cap E_2 .$ 

\sms

\noindent However, given that $B_1' = B_2'$, it follows by i) and symmetry that we also have
\begin{itemize}
\item[iii)] $D_{J_1} \cap E_2 \subandsupeq D_{J_2} \cap E_2.$ 
\end{itemize}

Since $B_2' \in \{ I_q, I_q^c \}$, it is an easy exercise to show that for $i=1,2,\ D_{J_i} \cap E_2 \ne E_2 .$  Since ii) and iii) both hold,  it follows that for some $i\in \{ 1,2 \},$  $D_{J_i} \cap E_2 = \emptyset.$  Without loss of generality, we may assume that this holds for $i=1.$
Clearly $B_2' \ne I_q$ for otherwise $\{ e_{2j}\in E_2 \ \big| \ j \in J_1 \} \subset D_{J_1} \cap E_2$.  Thus $B_2' = I_q^c.$  Since $k = |J|$ is even, $| J_1 |$ and $| J_2 |$ have the same parity.  Suppose $| J_1 |$ is even. If $j\in J_1,$  then $e_{2j} \not\in D_{j}^2$ and for all $i \in J_1 - j,$ $e_{2j} \in D_{i}^2.$  Thus $e_{2j} \in D_{J_1}^2$ since $| J_1 -j |$ is odd, a contradiction.  Suppose instead that $|J_1|$ is odd.  If $j\in J_2$ then it follows that for all $i \in J_1$, $e_{2j} \in D_{i}^2.$  It then follows that $e_{2j} \in D_{J_1}^2,$ since $| J_1 |$ is odd, a contradiction.  From the above, It follows that $B_1 = T_{\beta(q)}$ and $B_2= T_q.$
\end{proof}


\begin{noname}
If $k$ is even, then $B_1$ and $B_2$ are not of the same type.\label{nona-difftype}
\end{noname}

\begin{proof}
Suppose $k$ is even and $B_1$ and $B_2$ are of the same type.  By (\ref{nona-B1=B2}), $B_1 = T_{\beta(q)}$ and $B_2 = T_q.$  By assumption, the submatrix of $A_1'$ formed by the first $\beta(q)$ columns is $B_1.$  
We may assume, by permuting the columns of $A_2',$ that $B_2$ is the submatrix of $A_2'$ formed by rows $r_1 < r_2 < \dots  < r_q$ and columns $1, \dots ,q.$ 
Let $R = \{ r_1, \dots , r_q \}.$
 For $i=1,2,$ and $j= 1, \dots ,n_i$, let $e_{ij}$ be the element corresponding to the $j$'th column of $A_i'.$ Let $E_1 = \{ e_{11}, e_{12}, \dots ,e_{1\beta(q)} \}$ and let $E_2 = \{ e_{21}, \dots ,e_{2q} \}.$  For all $r\in R$ and $j\in [2],$ let $F_r^j = D_r \cap E_j$ and for all subsets $I \subseteq R,$ let $F_{I}^j = D_{I} \cap E_j.$  Given that $B_1 = T_{\beta(q)},$ it follows that for all $r\in R,$ $F_{r}^1  = \{ e_{11}, \dots ,e_{1r} \}$.  Since $B_2 = T_q$, we have that for all distinct $r,s\in R$, $F_{r}^2  \subandsup F_{s}^2.$
We define a poset on $R$ where for all distinct $r,s \in R,$ $r \prec s$ if $s < r$ and $F_{r}^2 \subset F_{s}^2.$  Let $\alpha$ be the size of the largest anti-chain in this poset. It follows by Dilworth's Theorem \cite{Dil}  that one can partition the poset into $\alpha$ chains.  Thus there is a chain of size at least $\frac q\alpha$. Since $\max \{ \alpha, \frac q\alpha \} \ge \sqrt{q},$ it follows that there is an anti-chain or a chain of size at least 
$\sqrt{q}.$  Assuming that $\sqrt{q} \ge k,$ there is a subset $S = \{ r_{i_1}, r_{i_2}, \dots ,r_{i_k} \} \subset R$ where $i_1 < i_2< \cdots < i_k$ and either $r_{i_1} \prec r_{i_2} \prec \cdots \prec r_{i_k}$ or $I$ is an anti-chain.  In the former case,
$F_{r_{i_1}}^2  \subset F_{r_{i_2}}^2 \subset \cdots \subset F_{r_{i_k}}^2$, and in the latter case, $F_{r_{i_1}}^2 \supset F_{r_{i_2}}^2 \supset \cdots \supset F_{r_{i_k}}^2.$
 In either case, we claim that $D_I + C_1$ is a circuit, leading to a contradiction.  Suppose that this is not the case. Then there is a partition $S_1 \dot{\cup} S_2$ of $S$ such that
 
 \vspace{0.1in}

i')\  $F_{S_1}^1  \subandsupeq F_{S_2}^1$ \ \ and \ \ ii')\ $F_{S_1}^2 \parallelbowtie F_{S_2}^2.$ 

\vspace{0.1in}

For $\ell=1,2$ and $j = 1, \dots ,k$ let $$a_{\ell j} = | S_\ell \cap \{ r_{i_j}, \dots ,r_{i_k} \} | \  \mathrm{and}\  b_{\ell j} =  | S_\ell \cap \{ r_{i_1}, \dots ,r_{i_j} \} | .$$
For $j=1, \dots ,k$ let $1 \le \beta_j \le q$ be such that $F_{r_{i_j}}^2 = \{ e_{21}, \dots ,e_{2\beta_i} \}.$ 
Since for all $j,j' \in [k],$ $e_{1r_{i_j}} \in F_{r_{i_{j'}}}^1$ iff $j' \ge j,$ it follows that for $\ell= 1,2$ and $j=1, \dots ,k,$ $e_{1r_{i_j}} \in F_{I_\ell}^1$ if and only if $a_{\ell j}$ is odd.  
Without loss of generality, we may assume that $r_{i_k} \in S_1.$  Then $a_{1k} =1$ and $a_{2k} =0$, implying that $e_{1r_{i_k}} \in F_{S_1}^1 - F_{S_2}^1.$  It now follows from i') that
$F_{S_1}^1 \supset F_{S_2}^1.$

 Suppose first that $F_{r_{i_1}}^2  \subset F_{r_{i_2}}^2 \subset \cdots \subset F_{r_{i_k}}^2$.  
 We may assume that $F_{r_{i_k}}^2 \ne E_2.$  Then $F_{S_1}^2 \cup F_{S_2}^2 \subset E_2$ and hence $F_{S_1}^2 \not\bowtie F_{S_2}^2.$  It follows that $F_{S_1}^2 \parallel F_{S_2}^2.$   Let $j\in [k]$ be such that $a_{2j}$ is odd. 
Then 
 $e_{1r_{i_j}} \in F_{S_2}^1.$  Furthermore, since $F_{S_2}^1 \subseteq F_{S_1}^1,$ it follows that $e_{1r_{i_j}} \in F_{S_1}^1$ and hence $a_{1j}$ is odd.  It now follows that $e_{2\beta_j} \in F_{S_1}^2 \cap F_{S_2}^2,$ a contradiction.
 

Suppose instead that $F_{r_{i_1}}^2 \supset F_{r_{i_2}}^2 \supset \cdots \supset F_{r_{i_k}}^2.$  We observe that for $i=1,2$ and $j= 1, \dots ,k,$ $e_{2\beta_j} \in F_{S_i}^2$ if and only if $b_{ij}$ is odd. We may assume that $F_{r_{i_1}}^2 \ne E_2.$ Then $F_{S_1}^2 \cup F_{S_2}^2 \subset E_2$ and hence $F_{S_1}^2 \not\bowtie F_{S_2}^2.$ 
It follows from ii') that $F_{S_1}^2 \parallel F_{S_2}^2.$  Note that since $k = |S|$ is even, $|S_1|$ and $|S_2|$ have the same cardinality.  Suppose $| S_1 |$ and $| S_2 |$ are both odd.  Then given that for $i=1,2$, $b_{ik} = | S_i |$, it would follow that $b_{1k}$ and $b_{2k}$ are both odd, implying that $e_{2\beta_k} \in F_{S_1}^2 \cap F_{S_2}^2,$  a contradiction. It follows that $| S_1 |$ and $| S_2 |$ are both even.
Suppose $r_{i_1} \in S_2.$  Since $| S_1 |$ and $| S_2 |$ are both even, it follows that $a_{12} = | S_1 |$ is even and $a_{22} = | S_2 | -1$ is odd, implying that $e_{1r_{i_2}} \in F_{S_2}^1 - F_{S_1}^1,$ a contradiction.  Thus it follows that $r_{i_1} \in S_1.$  Let $m$ be the largest integer such that $\{ r_{i_1}, \dots ,r_{i_m} \} \subset S_1.$  Suppose first that $m$ is odd.  Then $b_{1(m+1)} = m$ and $b_{2(m+1)} =1$ and hence $b_{1(m+1)}$ and $b_{2(m+1)}$ are both odd, implying that $e_{2\beta_{m+1}} \in F_{S_1}^2 \cap F_{S_2}^2,$ yielding
a contradiction.  Thus $m$ must be even.  Since $r_{i_{m+1}} \in S_2,$ it follows that $a_{1(m+2)} = | S_1 | -m$ which is even, and $a_{2(m+2)} = | S_2 | -1$, which is odd.  Thus $e_{1r_{i_{m+2}}} \in F_{S_2}^1 - F_{S_1}^1,$ yielding a contradiction.   

From the above, it follows that  i') and ii') can not both hold, yielding a contradiction.
\end{proof}

\section{The case where $B_1$ and $B_2$ are of different types}

By (\ref{nona-difftype}),  $B_1$ and $B_2$ are of different types (assuming $k$ is even).  We shall show that there exist cycles $C_i',\ i = 1,2$ as in S2).  
Recall the matrix $A'$ and matrices $A_i,\ i = 1,2$ defined in Section \ref{sec-matrices}.
Since $B_1$ and $B_2$ are of different types, it is seen that one can permute the rows of $A'$ and the columns of $A_1'$ and/or the columns of $A_2'$, so as to obtain a matrix $A''$ and submatrices $A_i'', \ i = 1,2$ corresponding to $A_i',\ i = 1,2$, where for $i=1,2$, the submatrix $B_i'$ formed by the first $q$ rows and $q$ columns of $A_i''$ has the same type as $B_i.$  We may assume that the elements of $C_1 \cup C_2$ and circuits $D_i, \ i \in [q]$ are indexed so that $e_{ij}$ is the element corresponding to the $j$'th column of $A_i''$ and $D_i$ corresponds to the $i$'th row of the matrix $A''$.

 %
 \subsection{$I$-triples}\label{sec-triple}
 
 By assumption, for all $I \in {\binom {[p]}k}$, $D_I + C_1$ is not a circuit.  Thus for all $I \in {\binom {[p]}k}$, we can assign a partition $I = I_1\dot{\cup}I_2$ to $I$ where $D_I + C_1 = G_{I_1} + G_{I_2},$ $G_{I_1}$ and $G_{I_2}$ being disjoint cycles for which $D_{I_i} \cap X = G_{I_i} \cap X, \ i =1,2.$  We shall adopt the convention that the least integer in $I$ belongs to $I_1$.  Let $H_I = D_{I_1} + G_{I_1},$ noting that $H_I \in \C (C_1 \cup C_2) = \{ \emptyset, C_1, C_2, C_1 \cup C_2 \}$ and $D_{I_2} + G_{I_2} = H_I + C_1.$
We shall refer to $(I_1, I_2,H_I)$ as the $I$-triple which is assigned to $I.$
Note that, as before, we have that $D_{I_1}^1  \subandsup D_{I_2}^1$ and $D_{I_1}^{2} \parallelbowtie D_{I_2}^{2}$.

In general, suppose $I \subseteq [p]$ and $I = S_1 \dot{\cup} S_2$ is a partition of $I$ where $D_I + C_1= G_{S_1} \dot{\cup} G_{S_2}$, $G_{S_1}, G_{S_2}$ being disjoint cycles such that for $i=1,2,$ $D_{S_i}\cap X = G_{S_i} \cap X$ and $D_{S_1} + G_1 = H$.  We refer to $(S_1, S_2)$ as an $\mathbf{I}${\bf-pair} and $(S_1, S_2, H)$ as an $\mathbf{I}$-{\bf triple}.

For all $I \in {\binom {[p]}k},$ we assign a binary vector $\bm{\tau}(I) = (a_1, \dots , a_k) \in \mathbb{Z}_2^k$ (called the {\bf sign} of $I$) as follows: Let $I = \{ i_1, \dots ,i_k \}$ where $i_1 < i_2 < \cdots < i_k.$  Then for $j=1, \dots ,k$, $a_j =1$ iff $i_j \in I_1$ (noting that by convention, $i_1 \in I_1$ and hence $a_1 =1$).

We also define a vector $\bm{\varsigma}(I) = (s_1, \dots ,s_{\ell})$ iteratively as follows:  Let $s_1$ be the maximum value of $i$ such that $a_1 = \cdots = a_i =1.$
Suppose $s_1, \dots ,s_j$ have been defined for some $j\ge 1.$  If $\sum_{i=1}^j s_i = k$, then we are done and $\ell = j.$  Otherwise, let $s_{j+1}$ be the maximum value of $i$ such that $a_{s_1 + \cdots + s_j +1} =  \cdots = a_{s_1+ \cdots + s_j +i}$ and continue with $j+1$ in place of $j.$  Roughly speaking, if one divides $\bm{\tau}(I)$ into blocks of $0$'s and $1$'s, then $s_j$ is the number of ``bits" in the $j$'th block and $\ell$ equals the number of blocks.

We have three important observations.

\begin{observation}
Suppose $k$ is even.  Then $B_1' = T_q$ and $B_2' \in \{ I_q, I_q^c \}.$\label{obs5.5}
\end{observation}
\begin{proof} Let  $I \in {\binom {[q]}k}$.  We have i)\ $D_{I_1}^1 \subandsup D_{I_2}^1$ and ii) $D_{I_1}^2 \parallelbowtie D_{I_2}^2$.
Now i) implies that $D_{I_1} \cap \{ e_{11}, \dots ,e_{1q} \} \subandsupeq D_{I_2} \cap \{ e_{11}, \dots ,e_{1q} \}$ and this in turn implies that $B_1' \ne I_q.$  Suppose $B_1' = I_q^c.$   Given that $k = |I|$ is even, it follows that $| I_1 |$ and $| I_2 |$ have the same parity. Let $i\in I_1.$  Then $e_{1i} \not\in D_i^1$ and $\forall j \in I_1 -i,$ $e_{1i} \in D_j^1$.  Moreover, for all $j\in I_2,$ $e_{1i} \in D_{j}^1.$ Thus if $| I_1|$ is even, $e_{1i} \in D_{I_1}^1 - D_{I_2}^1$ and if $|I_1|$ is odd, then $e_{1i} \in D_{I_2}^1 - D_{I_1}^1$.  However, the same observation applies for $I_2$ in place of $I_1.$ Thus $D_{I_1}^1 - D_{I_2}^1 \ne \emptyset$ and $D_{I_2}^1 - D_{I_1}^1 \ne \emptyset$ contradicting the fact that $D_{I_1}^1 \subandsup D_{I_2}^1.$   We conclude that $B_1' \ne I_p^c$ and hence $B_1' = T_p$ and $B_2' \in \{ I_p, I_p^c \}.$ 
%
 %
\end{proof}  

\begin{observation}
Suppose $k$ is even.  Then for all $I \in {\binom {[q]}k},$ $D_I + C_2$ and $D_I + C_1 + C_2$ are circuits. \label{obs6}
\end{observation}

\begin{proof}
Let $I \in {\binom {[q]}k}$ and suppose that $D_I + C_2$ is not a circuit.  Then it follows by Lemma \ref{lem-useful2} that there is a partition $I = S_1 \dot{\cup} S_2$ of $I$ for which $D_{S_1}^2 \subandsup D_{S_2}^2$.  
It then follows that  $D_{S_1} \cap \{ e_{11}, \dots ,e_{1q} \} \subandsupeq D_{S_2} \cap \{ e_{11}, \dots ,e_{1q} \}$.  By Observation \ref{obs5.5}, $B_2' \in \{ I_q, I_q^c \}$.  Since the above implies that $B_2' \ne I_q,$ it follows that $B_2' = I_q^c.$  Since $k$ is even, one can show using similar arguments as before that $D_{S_1}^2 - D_{S_2}^2 \ne \emptyset$ and $D_{S_2}^2 - D_{S_1}^2 \ne \emptyset,$ contradicting the fact that $D_{S_1}^2 \subandsup D_{S_2}^2.$  Thus $D_I + C_2$ is a circuit. 

To show that $D_I + C_1 + C_2$ is a circuit, we note that if this is not the case, then Lemma \ref{lem-useful2} implies that there is a partition $I = S_1 \dot{\cup} S_2$ of $I$ such that for $i=1,2,$ $D_{S_1}^i \subandsup D_{S_2}^i$.  In this case, $B_2' \ne I_p$ and thus $B_2' = I_q^c.$ However, one can now show as before that $D_{S_1}^2 - D_{S_2}^2 \ne \emptyset$ and $D_{S_2}^2 - D_{S_1}^2 \ne \emptyset,$ contradicting the fact that $D_{S_1}^2 \subandsup D_{S_2}^2.$  
\end{proof}

\begin{observation}If $k$ is even, then we may assume that $B_2' = I_q.$\label{obs6.1}
\end{observation}

To justify the above,
suppose $k$ is even.  Then for all $i\in [q],$ replace $D_i$ by the circuit $D_i' = D_i + C_2.$  Given that $k$ is even, it follows that for all $I \in {\binom {[p]}k},\ D_I' = \sum_{i\in I}D_i' = D_I.$  Thus for all $I \in {\binom {[p]}k},$ $D_I' + C_1$ is not a circuit and one can use all the previous arguments with $D_i'$ in place of $D_i.$  If $B_2' = I_q^c,$ then the corresponding matrix $B_2'$ for the circuits $D_i',\ i \in [q]$ will be $I_q.$  Because of this, we may assume that $B_2' = I_q.$  

\begin{lemma}
There is a constant $\xi(k,q)$, depending only on $k$ and $q$, such that provided $p \ge \xi,$ there is a $q$-subset $Q \subset [p]$ such that for all $I,J\in {\binom {Q}k},$ $\bm{\tau}(I) = \bm{\tau}(J)$ and $H_I = H_J.$
\label{lem-samesig}
\end{lemma}

\begin{proof}
We assign to each $k$-subset $I \in {\binom {[p]}k}$ the pair $(\bm{\tau}(I)), H_I)$, of which there are at most $2^{k+2}$ possible pairs.  It follows by Theorem \ref{the-Ram} that there is a constant 
$R^k(\overbrace{q, \dots ,q}^{2^{k+2}})$ such that provided $p \ge R^k(q, \dots ,q)$, there is a $q$-subset $Q \subset [p]$ such that all $k$-subsets of $Q$ are monochromatic; that is, for all $I, J \in {\binom {Q}k},$ $\bm{\tau}(I) = \bm{\tau}(J)$ and $H_I = H_J.$  Thus $\xi = R^k(q, \dots ,q)$ is the required constant.
\end{proof}

By the above, we shall assume that $Q \subseteq [p]$ is a $q$-subset such that for all $I \in {\binom {Q}k},$ $\bm{\tau}(I) = (a_1, \dots ,a_k), \ \varsigma(I) = (s_1, \dots ,s_{\ell})$ and $H_I = H.$ Furthermore, for convenience, we may assume that $Q = [q].$

\subsection{Finding disjoint sets}

To complete the proof of Theorem \ref{the-main}, we shall show that one can find cycles in $N$ satisfying S2) (with $N$ in place of $M$).   To do this, it will suffice to prove the next lemma. 
%

\begin{lemma}
Assuming $k$ is even, there are four pairwise disjoint sets $U_i \subset [q],\ i \in [4]$ such that  for all $1 \le i < j \le 4,$ $(U_i, U_j)$ is a $(U_i \cup U_j)$-pair and $|U_i \cup U_j| \ge k.$
\label{lem-disjointtrips}
\end{lemma}
In advance of the proof, we shall give a small example which illustrates how the subsets $U_i, \ i\in [4]$ in the lemma are constructed.
 \subsubsection{An example} 
Suppose that in the above, $k= 10$ and for all ${I \in \binom {[q]}{10}}$, $\varsigma(I) = (s_1, s_2, s_3) = (4,3,3).$   Then for all subsets $R_1 < R_2 < R_3$ of $[q]$ where $|R_i| = s_i, \ i = 1,2,3$, the triple of sets $(U_1, U_2,H)$, where $U_1 = R_1 \cup R_3$ and $U_2 = R_2$, 
is a $U_1 \cup U_2$-triple.  
Suppose now that we want to find four disjoint subsets $U_i \subset [q], \ i \in [4]$ such that for all $1 \le i < j \le 4,$ $(U_i, U_j, H)$ is $U_i \cup U_j$-triple where we only require that $|U_i \cup U_j| \ge 10.$  
To do this, we first use an observation that asserts that if subsets $R_1 < R_2 < R_3$ in $[q]$ are such that for $i = 1,2,3$, $|R_i| \ge s_i,$ and $|R_i|$ and $s_i$ have the same parity, then $(S_1, S_2, H)$ is a $S_1\cup S_2$-triple, where $S_1 = R_1 \cup R_3$ and $S_2 = S_2 = R_2.$  We shall exploit this observation in the next construction. 

Let $U_{1,1} < U_{2,1} < U_{3,1} < U_4 < U_{3.2} < U_{2,2} < U_{1,2}$ where $|U_{1,1}| = |U_{2,1}| = |U_{3,1}| =4,\ |U_4| = 7,$ and $|U_{1,2}| = |U_{2,2}| = |U_{3,2}| = 3.$  For $i= 1,2,3$ let $U_i = U_{i,1} \cup U_{i,2}.$
 Since for $i=1,2,3,$  $U_{i,1} < U_4 < U_{i,2}$, it follows from the above that for $i=1,2,3,$ $(U_i, U_4,H)$ is a $U_i \cup U_4$-triple.  Since for all $1 \le i < i' < 4,$  $U_{i,1} < U_{i',1} \cup U_{i',2} < U_{i,2}$ and $|U_{i'}| = |U_{i',1}| + |U_{i',2}| = 7,$ it also follows from the above that 
 $(U_i, U_{i'}, H)$ is an $U_i \cup U_{i'}$-triple. 

\subsubsection{A lemma for constructing pairs}

Suppose $I \in {\binom {[q]}k}.$  By assumption, $(I_1, I_2,H)$ is a $I$-triple and $\varsigma(I) = (s_1, \dots ,s_\ell).$  Let $\omega(0) = 0$ and for $i= 1, \dots ,\ell,$ let $\omega(i) = \sum_{j=1}^i s_j.$ Assuming that $I = \{ i_1, \dots ,i_k \}$ where $i_1 < i_2 < \cdots < i_k$, it follows that 
$I_1 = \bigcup_{j\in [\ell]^o} \{i_{\omega(j-1) +1}, \dots , i_{\omega(j)} \} $ and $I_2 =  \bigcup_{j\in [\ell]^e} \{ i_{\omega(j-1)+1}, \dots , i_{\omega(j)} \}$.
Our first task is to show that one can replace the sets $\{ i_{\omega(j-1)+1}, \dots , i_{\omega(j)} \}$ in $I_1$ or $I_2$ with larger sets so as to achieve sets $I_1', I_2'$ for which $(I_1', I_2')$ is a $I_1' \cup I_2'$-pair.
This gives us the needed flexibility when constructing the sets $U_1, \dots ,U_m$ described in Lemma \ref{lem-disjointtrips}.

Let $\Q$ be the set of all $\ell$-tuples $(Q_1, \dots, Q_\ell) \in \left(2^{[q]}\right)^\ell$ where $Q_1 < Q_2 < \cdots < Q_{\ell}$ and for all $i\in [\ell],$ $|Q_i| \ge s_i$ and $|Q_i|$ and $s_i$ have the same parity.

\begin{lemma}
For all $(Q_1, \dots ,Q_\ell) \in \Q,$ if $R_1 = \bigcup_{i\in [\ell]^o}Q_i,$ $R_2 =  \bigcup_{i\in [\ell]^e}Q_i$ and $R= R_1 \cup R_2$, then
$(R_1,R_2,H)$ is an $R$-triple.
\label{lem11}
\end{lemma}

\begin{proof}
We shall use induction on $\sum_{i=1}^\ell |Q_i|.$  When $\sum_{i=1}^\ell |Q_i| = \sum_{i=1}^\ell s_i$, that is, for all $i$, $|Q_i| = s_i,$ the lemma is true by assumption.  Suppose that the lemma is true for all $(Q_1, \dots ,Q_\ell) \in \Q$ where 
$\sum_{i=1}^\ell |Q_i| = \alpha\ge \sum_{i=1}^{\ell}s_i.$  Let $(Q_1, \dots ,Q_\ell) \in \Q$ where $\sum_{i=1}^\ell |Q_i| = \alpha + 2.$ Let $R_1 = \bigcup_{i\in [\ell]^o}Q_i,$ $R_2 =  \bigcup_{i\in [\ell]^e}Q_i$ and $R= R_1 \cup R_2.$
For some $i,$ $|Q_i| \ge s_i +2$ and without loss of generality we may assume this is true for $i=1.$  Let $Q_1 = \{ q_1, \dots ,q_t \}$ where $q_1 < q_2 < \cdots < q_t$.
For $j= 0,1,2,$ let $Q_1^j = Q_1 -  \{ q_{t -2}, q_{t-1}, q_t \} + q_{t -j}$ and  $R_1^j = (R_1 - Q_1) \cup Q_1^j$ and $R^j = R_1^j \cup R_2.$  For $j=0,1,2$ we have $Q_1^j < Q_2 < \cdots  < Q_{\ell}$ and $(Q_1^j, Q_2, \dots ,Q_\ell) \in \Q.$  It follows by the inductive assumption that for $j= 0,1,2,$ $(R_1^j, R_2, H)$ is an $R^j$-triple.  For $j=0,1,2,$ let $G_1^j$ and $G_2^j$ be pairs of disjoint cycles where for $j= 0,1,2,$ $G_1^j + G_2^j = D_{R^j} + C_1$ and $G_1^j \cap X = R_1^j,\ G_2^j \cap X = R_2,$ $G_1^j + D_{R_1^j} = H.$
Then for $j=0,1,2,$ $G_2^j + D_{R_2} = H + C_1$ and hence $G_2^0 = G_2^1 = G_2^2 = D_{R_2} + H + C_1.$ 
In addition, we have that  $$\sum_{j=0}^2 G_1^j + G_2^0  = \sum_{j=0}^2 (G_1^j + G_2^j) = \sum_{j=0}^2(D_{R^j} + C_1) =  D_R + C_1.$$  Now $D_R + C_1$ is the disjoint union of cycles  $G_1 = \sum_{j=0}^2 G_1^j$ and $G_2 = G_2^0$ where for $i = 1,2,$ $G_i \cap X = R_i.$  We also have that $G_2 + D_{R_2} = G_2^0 + D_{R_2} = H + C_1$ and thus $(R_1, R_2, H)$ is an $R$-triple.  The proof now follows by induction.  
\end{proof}

\subsubsection{Balanced sequences} 

We say that a finite sequence of integers $b_1, b_2, \dots ,b_m$ is {\bf balanced} if\\

\begin{itemize}
\item[B1)] $b_{m} - b_{m -1} + \cdots + (-1)^{m +1} b_1=0$ and 
\item[B2)] for $i = 1, \dots ,m-1$, $b_i - b_{i-1} + \cdots + (-1)^{i+1}b_1 \ge 0.$
\end{itemize}

Assume that $k$ is even and let $I \in {\binom {[q]}k}$ where $\varsigma(I) = (s_1, \dots ,s_\ell).$ 
Let $t_i,\ i \in [\ell],$ be positive integers chosen such that 
for $i=1, \dots, \ell$,  $t_i \ge s_i$ and $s_i$ and $t_i$ have the same parity.  Given that $\sum_i s_i =k$ is even, it is a straightforward exercise to show that $t_1, \dots ,t_{\ell}$ can be chosen so that it is a balanced sequence.   We shall define integers $\mu_i, \ i = 0,1, \dots ,\ell$ as follows: let $\mu_0 := 0,$ and for $i=1, \dots , \ell,$ $\mu_i := t_i - t_{i-1} + \cdots +  (-1)^{i+1}t_1.$  Since $t_1, \dots ,t_{\ell}$ is balanced, B1) implies that $\mu_\ell =0,$ and B2) implies that for all $i\in [\ell -1],$ $\mu_i \ge 0.$  The sequence $\mu_i, \ i = 0, \dots ,\ell$ has the favourable property that 
$i=1, 2, \dots ,\ell$, $\mu_{i-1} + \mu_i = t_i$.
As a consequence of Lemma \ref{lem11}, we have the following:  

\begin{observation}
Let $Q_1 < Q_2 < \cdots < Q_{\ell}$ be disjoint subsets of $[q]$, where for all $i\in [\ell], \ |Q_i| = t_i.$  Then for $R_1 = \bigcup_{i\in [\ell]^o}Q_i,$ $R_2 =  \bigcup_{i\in [\ell]^e}Q_i$, and $R = R_1 \cup R_2$, we have that $(R_1,R_2,H)$ is an $R$-triple.\label{obs-appl}
\end{observation}

\subsubsection{Proof of Lemma \ref{lem-disjointtrips}}\label{sec-disjointtrips}

We shall complete the proof of Lemma \ref{lem-disjointtrips}.  We shall assume that $k$ is even. For convenience, we shall assume $\ell$ is even; the construction in the case where $\ell$ is odd being very similar.

Our objective is to construct pairwise disjoint sets $U_i \subset [q], \ i \in [4]$ such that  for all $1 \le i < j \le 4,$ $(U_i, U_j, H)$ is a $(U_i \cup U_j)$-triple and $|U_i \cup U_j| \ge k.$  Throughout, we may assume that $q$ is as large as needed.
To begin with, for all $i\in \{ 1, 4 \},$ $U_i$ will be the disjoint union of subsets $\displaystyle{U_i = \bigcup_{j=1}^{ \frac {\ell}2}U_{i,j}}$ and 
for all $i\in \{ 2,3 \},$ $U_i$ will be the disjoint union of subsets $\displaystyle{U_i = \bigcup_{j=1}^{\ell}U_{i,j}}.$

The sizes of the sets $U_{i,j}$ are defined such that for $j= 1, \dots , \frac {\ell}2,$ $|U_{1,j}| = t_{2j-1},$ $|U_{4, j}| = t_{2j},$ and for all $i \in \{ 2, 3 \}$ and $j\in [\ell]$, $|U_{i,j}| = \mu_j$
Note that since $\mu_\ell = 0,$ we have that for all $i \in \{ 2,3 \},$ $U_{i,\ell} = \emptyset.$  Moreover, for all $i\in \{ 2,3 \}$ and $j\in [2,\ell], \ |U_{i,j-1} \cup U_{i,j}| = \mu_{j-1} + \mu_j = t_j.$
In addition, the sets $U_{i,j}$ are chosen such that: 
\begin{itemize}
\item $U_{1,1} < U_{4,1} < U_{1,2} < U_{4,2} < U_{1,3} < U_{4,3} < \cdots < U_{1,\frac {\ell}2} < U_{4, \frac {\ell}2}.$
\item For all  $j \in [\ell]^o$, $U_{1, \frac {j+1}2} < U_{2,j} < U_{3,j} < U_{4, \frac {j+1}2}.$
\item  For all  $j \in [\ell-2]^e$, $U_{4, \frac j2}< U_{3,j} < U_{2,j} <U_{1,\frac j2 +1}$
\end{itemize}

For example, 
in the case where $\ell =4$, the sets $U_{i,j}$ are such that
$$U_{1,1} < U_{2,1} < U_{3,1} < U_{4,1} < U_{3,2} < U_{2,2} < U_{1,2} < U_{2,3} < U_{3,3} < U_{4,2}.$$

%
%
Since $U_{1,1} < U_{4, 1} < U_{1,2} < U_{4, 2} < \cdots < U_{1,\frac {\ell}2} < U_{4, \frac {\ell}2}$ and for $j=1, \dots ,\frac {\ell}2,$ $|U_{1,j}| = t_{2j-1}$ and $|U_{4,j}| = t_{2j},$
it follows from Observation \ref{obs-appl} that $(U_1, U_{4}, H)$ is a $U_1 \cup U_4$-triple.  We also see that $|U_1 \cup U_4| = \sum_i t_i \ge \sum_i s_i = k.$
We have that $$U_{2,1} < U_{3,1} \cup U_{3, 2} < U_{2,2} \cup U_{2,3} < U_{3,3} \cup U_{3,4} < \cdots < U_{2, \ell-2} \cup U_{2, \ell-1} < U_{3, \ell-1} .$$
Moreover, since $|U_{2,1}| = t_1,$ and for $j = 1,2, \dots ,\frac {\ell}2$, and $$|U_{2, 2j} \cup U_{2,2j+1}| = \mu_{2j} + \mu_{2j+1} = t_{2j+1},\ \ |U_{3,2j-1} \cup U_{3, 2j}| = \mu_{2j-1} + \mu_{2j} = t_{2j},$$ it follows from Observation \ref{obs-appl} that $(U_2, U_{3}, H)$ is a $(U_2 \cup U_{3})$-triple.  Moreover, $|U_2 \cup U_3| = \sum_i t_i \ge k.$

Let $i\in \{2,3 \}.$ Then we have $$U_{1,1} < U_{i,1} \cup U_{i,2} < U_{1,2} < U_{i,3} \cup U_{i,4} < \cdots < U_{1, \frac {\ell}2} < U_{i,\ell-1}.$$
Since for all $j\in [\frac {\ell}2],$ $|U_{1,j}| = t_{2j-1}$ and $|U_{i, 2j-1} \cup U_{i, 2j}| = t_{2j}$, it follows from Observation \ref{obs-appl} that $(U_1, U_i, H)$ is a $(U_1 \cup U_i)$-triple.  Moreover, $|U_1 \cup U_i| = \sum_i t_i \ge k.$
Similarly, we have $$U_{i,1} < U_{4,1} < U_{i,2} \cup U_{i,3} < U_{4,2} < U_{i,4} \cup U_{i,5} < \cdots < U_{i, \ell-2} \cup U_{i, \ell -1} < U_{4, \frac {\ell}2}.$$
Since $|U_{i,1}| = t_1,$ $|U_{4,1}| = t_2,$ and for all $j\in [2, \frac {\ell}2]$, $|U_{i, 2j-2} \cup U_{i, 2j-1}| = t_{2j-1}$ and $|U_{4,j}| = t_{2j},$ it follows  that $(U_i, U_4, H)$ is a $(U_i \cup U_4)$-triple where $|U_i \cup U_4| \ge k.$

From the above, $U_i,\ i \in [4]$ are seen to be the desired sets.  This completes the proof of Lemma \ref{lem-disjointtrips}.


\section{Completing the proof of Theorem \ref{the-main}}

We may assume that $k$ is even.  Recall that by Observations \ref{obs5.5} and \ref{obs6.1},  $B_1' = T_p$ and $B_2' = I_p$.  By Lemma \ref{lem-disjointtrips}, there are pairwise disjoint subsets $U_i \subset [q],\ i \in [4]$ such that for all $1 \le i < j \le 4,$ $(U_i,U_j,H)$ is an $U_i \cup U_j$-triple and $|U_i \cup U_j| \ge k.$  
Thus we have:

\begin{itemize}
\item[A)] For all $1 \le i < j \le 4,$ $D_{U_i}^1 \subandsup D_{U_j}^1$ and

\item[B)] for all $1 \le i < j \le 4,$ $D_{U_i}^2 \parallelbowtie D_{U_j}^2$.
\end{itemize}

It follows by (A) that the sets $D_{U_i}^1,\ i \in [4]$ are totally ordered by inclusion.  Thus we may assume that for all $1 \le i < j \le 4,$ $D_{U_i}^1 \subset D_{U_j}^1.$
%
%
Let $V_1 = U_1 \cup U_2$ and let $V_2 = U_{3} \cup U_4$.  Then $D_{V_1} = D_{U_1} + D_{U_2}$ and $D_{V_2} = D_{U_{3}} + D_{U_4}.$ We have that $D_{V_1}^1 \parallel D_{V_2}^1.$   Furthermore, by B), it follows that $D_{V_1}^2 \parallel D_{V_2}^2$ and hence $D_{V_1}$ and  $D_{V_2}$ are disjoint cycles.   Suppose $D_{V_1} + C_1 + C_2$ is not a circuit.  Then Lemma \ref{lem-useful2} implies that there is a partition $V_1 = S_1 \dot{\cup} S_2$ such that for $i =1,2,$ $D_{S_1}^i \subandsup D_{S_2}^i.$  However, given that $B_2' = I_q$, it follows that $D_{S_1}^2 \not\subandsup D_{S_2}^2.$  Thus $D_{V_1} + C_1 + C_2$ is a circuit and the same applies to $D_{V_2} + C_1 + C_2.$
Since for $i=1,2$, $|V_i| \ge k,$ we now see that the circuits $C_i' = D_{V_i},\ i = 1,2$ satisfy S2) (with $N$ in place of $M$).
This completes the proof.

\end{document}